\documentclass{amsart}
\usepackage{tikz, amsmath, amscd, amssymb, amsfonts, stmaryrd, vmargin, amsthm, soul, bm, graphicx, float, exscale, relsize, etoolbox, mathrsfs, url}
\usepackage[utf8]{inputenc}

\theoremstyle{plain}        \newtheorem{thm}{Theorem}

\theoremstyle{plain}        \newtheorem{pro}[thm]{Proposition}

\theoremstyle{plain}        \newtheorem{lem}[thm]{Lemma}

\theoremstyle{plain}        \newtheorem{cor}[thm]{Corollary}

\theoremstyle{plain}        \newtheorem{rem}[thm]{Remark}

\theoremstyle{plain}        \newtheorem{eg}{Example}

\theoremstyle{plain}

\theoremstyle{plain}

\numberwithin{equation}{section}
\numberwithin{thm}{section}

\usepackage{comment}
\usepackage[utf8]{inputenc}
\usepackage[english]{babel}
 
\usepackage{xcolor}

\begin{document}

\title{Uniform estimates of the Cauchy-Riemann equation on product domains
} 

\author{Yuan Yuan} 

\address{Department of Mathematics, Syracuse University, Syracuse, NY 13244, USA}
\email{yyuan05@syr.edu}

\begin{abstract}
 We observe that the continuity assumption on $f$ for the uniform estimates of  the canonical solution to $\bar\partial u = f$ on products of $C^2$ bounded planar domains in \cite{DPZ} can be reduced to the boundedness assumption. This completely
 answers the original question raised by Kerzman in 1971. 
 Moreover, the $L^p$ estimates of $\bar\partial$ is obtained for all $p \in [1, \infty]$.
\end{abstract}
\maketitle

\section{Introduction}
The estimates of  the $\bar\partial$-equation in various function spaces are one of the most important problems in several complex variables and partial differential equations and have invaluable applications in differential geometry, algebraic geometry and other subjects (cf. \cite{BS, ChS, FoK, H, O, R1, St2}).
In this paper, we aim to answer the following fundamental question of the uniform estimates of the $\bar\partial$-equation raised by Kerzman in 1971 (cf. \cite{K1} pp. 311-312). This question recently was  reminisced in \cite{CM1} (cf. lines 22-24 on page 409) and attracted substantial attentions.

\medskip

{\bf Remark} in \cite{K1} (cf. pp.311-312):
We do not know whether Grauert-Lieb's and Henkin's theorem holds in polydiscs, i.e., whether there exists a bounded solution $u$ to $\bar\partial u=f$ on $\Delta^n$ whenever $f$ is bounded in $\Delta^n$, $\bar\partial f = 0$.

\medskip

On bounded strictly pseudoconvex domains in $\mathbb{C}^n$, the uniform estimate for $\bar\partial$-equation is obtained by  Grauert-Lieb \cite{GL} and Henkin \cite{H2} in 1970. Sibony later constructed a smooth bounded weakly pseudoconvex domain in $\mathbb{C}^3$ and a $\bar\partial$-closed $(0,1)$-form $f$, continuous on the closure of the domain, such that every solution of $\bar\partial u=f$ is unbounded \cite{S} (cf. \cite{B, FS} for examples in $\mathbb{C}^2$).  In 1986 Forn{\ae}ss proved uniform estimates  for a class of pseudoconvex domains in $\mathbb{C}^2$, which  include the Kohn-Nirenberg example \cite{F}. The uniform estimates for  finite type domains in $\mathbb{C}^2$ and convex, finite type domains in $\mathbb{C}^n$ are solved by Fefferman-Kohn \cite{FeK} (cf. \cite{R1} as well) and Diederich-Fischer-Forn{\ae}ss \cite{DFF}. More recently, Grundmeier-Simon-Stens{\o}nes proved the uniform estimates for a wide class of finite type pseudoconvex domains in $\mathbb{C}^n$, including the bounded, pseudoconvex domains with real-analytic boundary \cite{GSS}. The uniform estimates for $\bar\partial$-equation has been an attractive problem for many authors and we refer the interested readers to \cite{FLZ, H1, K1, R2, RS} and references therein for detailed account of the subject and related problems.

\medskip

On the other hand, when the domain is not smooth, in particular, a product domain, the problem becomes somehow different. In 1971, Henkin obtained the uniform estimates for $\bar\partial u=f$ on the bidisk provided that $f$ is $C^1$ up to the boundary \cite{H3}. Landucci in 1975 proved the uniform estimates for the canonical solution on the bidisc provided that $f$ is $C^2$ up to the boundary \cite{L}. More recently, a very useful new solution integral operator was used by Chen-McNeal  to prove many interesting results on product domains, including the $L^p$ estimates for $\bar\partial$-equation  \cite{CM1, CM2}, and also by Fassina-Pan to prove the uniform estimates for $\bar\partial$-equation on the high dimensional product of planar domains \cite{FP}. Dong-Pan-Zhang later also apply the integral operator to further obtain the uniform estimates for the canonical solution to $\bar\partial u=f$ on the product of planar domains by assuming $f$ is merely continuous up to the boundary \cite{DPZ}. This is not only a deep result but also the proof contains fascinating ideas by   
combining the above-mentioned new integral operator and Kerzman's celebrated estimates of the Green function \cite{K2} as observed in \cite{BL}. For more related studies of the $\bar\partial$-equation on product domains, the interested readers may refer to \cite{CS, JY, PZ, YZZ, Z} and references therein.


\medskip

Heavily relying on the ideas developed in \cite{DPZ}, we are able to answer Kerzman's original question. The key difference is that in \cite{DPZ}, the canonical solution is re-written using the integral formula against $f$, which contains certain boundary integrals. This requires $f$ to be at least continuous to make sense of the boundary integral. However, we observe that the boundary integrals actually do not appear because the integral kernels vanish on the boundary. This already provides $L^p$ estimates when $f$ is sufficiently smooth, combining the deep estimates of the Green function by Kerzman. When $f$ is merely $L^p$, the estimates is achieved by using approximation as in \cite{DPZ}.

\medskip

 The main theorem of the paper is the following $L^p$ estimates of $\bar\partial$.

\begin{thm}\label{main}
Let $\Omega = D_1 \times \cdots \times D_n$, where, for each $1 \leq j \leq n$, $D_j$ is a $C^2$-smooth bounded planar domain. 
For any $p \in [1, \infty]$, assume $f \in L^p_{(0, 1)}(\Omega)$.
Then there exists a constant $C_{\Omega}>0$ (independent of $p$) such that
 the canonical solution to
 $\bar\partial u = f$ satisfies $\|{\bf T} f\|_{L^p} \leq C_{\Omega} \|f\|_{L^p}$.
\end{thm}

In particular, when $\Omega = \Delta^n$ is the polydisc, this answers Kerzman's question.

\begin{cor}\label{cor}
There exists a constant $C>0$ such that
for any $f \in L^\infty_{(0, 1)}(\Delta^n)$ with $\bar\partial f =0$, the canonical solution to
 $\bar\partial u = f$ satisfies $\|u\|_\infty \leq C \|f\|_\infty$.
\end{cor}


\section{Proof of the Theorem}
The majority of the proof was already carried out by Dong-Pan-Zhang in \cite{DPZ}. We try to make the argument here as self-contained as possible.

\subsection{One dimensional case}
We follow the argument by Barletta-Landucci \cite{BL} and Dong-Pan-Zhang in \cite{DPZ}. Let $D \subset \mathbb{C}$ be a bounded planar with $C^2$ boundary and $H(w, z) = \frac{1}{2\pi i (z-w)}$ be the Cauchy kernel on $D$. The following kernel is defined by Barletta-Landucci \cite{BL},
$$S(w, z) = L(w. z) - H(w, z),$$
where for any $w \in D$, $L(w, z)$ solves the Dirichlet problem 
\begin{equation*}
\left\{
\begin{aligned}
\Delta L(w, z) &=0, \quad z\in D; \\
L(w, z) &=H(w, z), \quad z \in \partial D.
\end{aligned}
\right.
\end{equation*}
It is known that for fixed  $w \in D$,  $L(w, z) \in C^{1, \alpha}(\overline{D})$ for any $\alpha \in (0, 1)$ (cf. \cite{DPZ}). 
Define ${\bf T} f (w) = \int_D S(w, z)f(z) d\bar z \wedge d z$. Then ${\bf T} $ is the canonical solution operator for $\bar\partial u = f d\bar z$ on $D$ (cf. Theorem in \cite{BL} or Proposition 2.3 in \cite{DPZ}).
 
\medskip

Here is the {\it key observation}: $S(w, z) = 0$ for any $w \in D, z \in \partial D$. 

\begin{eg}
Let $D=\Delta$ be the unit disc in $\mathbb{C}$. Then it is very easy to verify that $L(w, z)= \frac{\bar z}{2\pi i(1-w \bar z)}$ and thus $S(w, z) = \frac{1-|z|^2}{2\pi i (1-w \bar z)(w-z)}$ vanishes for $z \in \partial \Delta, w \not= z $.
\end{eg}

\subsection{High dimensional case when $f$ is sufficiently smooth}

Let $\Omega = D_1 \times \cdots \times D_n \subset \mathbb{C}^n$ be a bounded domain, where each $D_j$ is a bounded planar domain with  $C^2$ boundary, $1 \leq j \leq n$.
We first handle the case when $f$ is sufficiently smooth.
Assume $f = \sum_{j=1}^n f_j d\bar z_j$ to be a $\bar\partial$-closed $(0, 1)$-form on $\Omega$ with $f_j \in C^{n-1}(\overline{\Omega})$, $1 \leq j \leq n$, and define 
\begin{equation}
{\bf T} f = \sum_{s=1}^n (-1)^{s-1} \sum_{1 \leq i_1 \leq \cdots \leq i_s \leq n} {\bf T}_{i_1} \cdots {\bf T}_{i_s} \left( \frac{\partial^{s-1} f_{i_s}}{\partial \bar z_{i_1} \cdots \partial \bar z_{i_{s-1}}} \right),
\end{equation}
where ${\bf T}_j$ is the canonical solution operator on $D_j $.  
Let $$B_{1, 2, \cdots, k}(w, z) = \sum_{j=1}^k \prod_{m\not=j, m=1}^k |w_m - z_m|^2, $$
$$e^{1, 2, \cdots, k}_j(w, z)= \left( \prod_{l=1}^k S_l(w_l, z_l)\right)\frac{\prod_{m\not=j, m=1}^{k} |w_m - z_m|^2}{B_{1, 2, \cdots, k}(w, z)}  .$$ Then $ \sum_{j=1}^k e^{1, 2, \cdots, k}_j = \prod_{l=1}^k S_l(w_l, z_l)$.
Here is the key lemma. 

\begin{lem}\label{ibp}
 Then 
\begin{equation}
\int_{D_1 \times \cdots \times D_k} e^{1, 2, \cdots, k}_k(w, z) \frac{\partial^{k-1} f_k(z)}{\partial \bar z_1 \cdots \partial \bar z_{k-1}} dV(z)  = \int_{D_1 \times \cdots \times D_k} (-1)^k  f_k(z) \frac{\partial^{k-1} e^{1, 2, \cdots, k}_k(w, z)}{\partial \bar z_1 \cdots \partial \bar z_{k-1}} dV(z).
\end{equation} 
\end{lem}

\begin{proof}
Fix $w \in D_1 \times \cdots \times D_k$. Let $A_j^\epsilon = D_j \setminus \Delta_\epsilon(w_j)$ for $0 < \epsilon <<1$.
Note that for any $0 \leq m \leq k-2, 1\leq i_1 < \cdots < i_m \leq k-2$ and $l \not \in \{i_1 , \cdots, i_m, k\}$, $\frac{\partial^m e^{1, 2, \cdots, k}_k(w, z)}{\partial \bar z_{i_1} \cdots \partial \bar z_{i_m}} = 0 $ on $ D_1 \times \cdots \times \partial D_l \times \cdots \times D_k$ and $\frac{\partial^m e^{1, 2, \cdots, k}_k(w, z)}{\partial \bar z_{i_1} \cdots \partial \bar z_{i_m}} \rightarrow 0$ uniformly for $z \in D_1 \times \cdots \times \partial \Delta_\epsilon(w_l) \times \cdots \times D_k$ as $\epsilon \rightarrow 0$.

 Applying the Stokes formula to 
$$e^{1, 2, \cdots, k}_k(w, z) \frac{\partial^{k-2} f_k(z)}{\partial \bar z_2 \cdots \partial \bar z_{k-1}} dz_1 \wedge d \bar z_2 \wedge d z_2 \wedge \cdots \wedge d \bar z_{k} \wedge dz_k$$
on $A_1^\epsilon \times D_2 \times \cdots \times D_k$, we have
\begin{equation*}
\begin{split}
&~~~\int_{A_1^\epsilon \times \cdots \times D_k} \left( e^{1, 2, \cdots, k}_k(w, z) \frac{\partial^{k-1} f_k(z)}{\partial \bar z_1 \cdots \partial \bar z_{k-1}} + \frac{\partial e^{1, 2, \cdots, k}_k(w, z)}{\partial \bar z_1} \frac{\partial^{k-2} f_k(z)}{\partial \bar z_2 \cdots \partial \bar z_{k-1} } \right)  d \bar z_1 \wedge d z_1 \wedge \cdots \wedge d \bar z_{k} \wedge dz_k \\
&= \left( \int_{\partial D_1 \times D_2 \times\cdots\times D_k} -  \int_{\partial\Delta_\epsilon(w_j) \times D_2 \times\cdots\times D_k} \right) e^{1, 2, \cdots, k}_k(w, z) \frac{\partial^{k-2} f_k(z)}{\partial \bar z_2 \cdots \partial \bar z_{k-1}} dz_1 \wedge d \bar z_2 \wedge d z_2 \wedge \cdots \wedge d \bar z_{k} \wedge dz_k \\
&= -  \int_{\partial\Delta_\epsilon(w_j) \times D_2 \times\cdots\times D_k} e^{1, 2, \cdots, k}_k(w, z) \frac{\partial^{k-2} f_k(z)}{\partial \bar z_2 \cdots \partial \bar z_{k-1}} dz_1 \wedge d \bar z_2 \wedge d z_2 \wedge \cdots \wedge d \bar z_{k} \wedge dz_k,
\end{split}
\end{equation*}
as $e^{1, 2, \cdots, k}_k(w, z) =0 $ for $z_1 \in \partial D_1$.
As $\epsilon \rightarrow 0$, $e^{1, 2, \cdots, k}_k(w, z) \rightarrow 0$ uniformly for $z \in \Delta_\epsilon(w_j) \times D_2 \times\cdots\times D_k$. It follows that 
$$\int_{D_1 \times \cdots \times D_k}  e^{1, 2, \cdots, k}_k(w, z) \frac{\partial^{k-1} f_k(z)}{\partial \bar z_1 \cdots \partial \bar z_{k-1}} dV(z)= - \int_{D_1 \times \cdots \times D_k} \frac{\partial e^{1, 2, \cdots, k}_k(w, z)}{\partial \bar z_1} \frac{\partial^{k-2} f_k(z)}{\partial \bar z_2 \cdots \partial \bar z_{k-1} } dV(z) .$$
The lemma then follows by applying the argument repeatedly in all other variables $z_2, \cdots, z_{k-1}$.
\end{proof}

The next result provides $L^p$-estimates for $\bar\partial u =f$ when $f$ is smooth, where the key is an important application of the Kerzman's estimates of the Green function proved in \cite{DPZ} (cf. Proposition 4.5). Since the integrals here do not involve the boundary terms, the corresponding estimates become simpler.

\begin{pro}\label{smooth}
For $f \in C^{n-1}_{(0, 1)}(\overline{\Omega})$ with $\bar\partial f =0$, let $${\bf K} f(w)= \sum_{s=1}^n \sum_{1 \leq i_1 \leq \cdots \leq i_s \leq n} \sum_{j=1}^{s} \int_{D_{i_1} \times \cdots \times D_{i_s}} f_{i_j}(z)  \frac{\partial^{s-1} e^{i_1, \cdots, i_s}_{i_j}(w, z)}{\partial \bar z_{i_1} \cdots \partial \bar z_{i_{j-1}} \partial \bar z_{i_{j+1}} \cdots \partial \bar z_{i_{s-1}}} d\bar z_{i_1} \wedge \cdots dz_{i_s} .$$ Then ${\bf T} f  = {\bf K} f  .$
Moreover, for any $p \in [1, \infty]$,
there exists a universal constant $C_{\Omega}>0$ (independent of $p$) such that
 $u={\bf T} f$ is the canonical solution to
 $\bar\partial u = f$ and satisfies $\|{\bf T} f\|_{L^p} \leq C_{\Omega} \|f\|_{L^p}$.
\end{pro}

\begin{proof}
Since $f$ is  $\bar\partial$-closed, $\frac{\partial^{s-1} f_{i_1}}{\partial \bar z_{i_2} \cdots \partial \bar z_{i_s}} = \frac{\partial^{s-1} f_{i_2}}{\partial \bar z_{i_1} \partial \bar z_{i_3} \cdots \partial \bar z_{i_s}} = \cdots = \frac{\partial^{s-1} f_{i_s}}{\partial \bar z_{i_1} \cdots \partial \bar z_{i_{s-1}}}$. 
It follows that
 $$\left( \frac{\partial^{s-1} f_{i_s}(z)}{\partial \bar z_{i_1} \cdots \partial \bar z_{i_{s-1}}} \right) \prod_{j=i_1}^{i_s} S_j(w_j, z_j) = \sum_{j=1}^{s} e^{i_1, \cdots, i_s}_{i_j}(w, z) \frac{\partial^{s-1} f_{i_j}(z)}{\partial \bar z_{i_1} \cdots \partial \bar z_{i_{j-1}} \partial \bar z_{i_{j+1}} \cdots \partial \bar z_{i_{s-1}}}$$
 and thus by Lemma \ref{ibp},
$${\bf T}_{i_1} \cdots {\bf T}_{i_s} \left( \frac{\partial^{s-1} f_{i_s}}{\partial \bar z_{i_1} \cdots \partial \bar z_{i_{s-1}}} \right)=  (-1)^{s-1} \sum_{j=1}^{s} \int_{D_{i_1} \times \cdots \times D_{i_s}} f_{i_j}(z)  \frac{\partial^{s-1} e^{i_1, \cdots, i_s}_{i_j}(w, z)}{\partial \bar z_{i_1} \cdots \partial \bar z_{i_{j-1}} \partial \bar z_{i_{j+1}} \cdots \partial \bar z_{i_{s-1}}} d\bar z_{i_1} \wedge \cdots dz_{i_s} .
$$
Moreover, we have 
$${\bf T} f  = \sum_{s=1}^n \sum_{1 \leq i_1 \leq \cdots \leq i_s \leq n} \sum_{j=1}^{s} \int_{D_{i_1} \times \cdots \times D_{i_s}} f_{i_j}(z)  \frac{\partial^{s-1} e^{i_1, \cdots, i_s}_{i_j}(w, z)}{\partial \bar z_{i_1} \cdots \partial \bar z_{i_{j-1}} \partial \bar z_{i_{j+1}} \cdots \partial \bar z_{i_{s-1}}} d\bar z_{i_1} \wedge \cdots dz_{i_s} .$$
By Proposition 4.5 in \cite{DPZ}, $ \frac{\partial^{s-1} e^{i_1, \cdots, i_s}_{i_j}(w, z)}{\partial \bar z_{i_1} \cdots \partial \bar z_{i_{j-1}} \partial \bar z_{i_{j+1}} \cdots \partial \bar z_{i_{s-1}}} \in L^1(D_{i_1} \times \cdots \times D_{i_s})$ and the $L^1$ norm is independent of $w$.
Therefore, $L^p$ estimates $\|{\bf T} f\|_{L^p} \leq C_{\Omega} \|f\|_{L^p}$ follows from Young's convolution inequality. Furthermore, It follows from Theorem 4.3 in \cite{DPZ} that ${\bf T} f$ is the canonical solution.
\end{proof}

\subsection{High dimensional case when $f$ is $L^p$}
In this section, we will use approximation to handle the case when $f \in L^p_{(0, 1)}(\Omega)$
 for $p \in [1, \infty]$. In fact, we will show the following general $L^p$ estimates.
 
 \begin{thm}
For any $p \in [1, \infty]$, assume $f \in L^p_{(0, 1)}(\Omega)$. Then $u = {\bf K} f$ is the canonical solution to $\bar\partial u = f$ and satisfies $\|{\bf K} f\|_{L^p} \leq C_{\Omega} \|f\|_{L^p}$.
\end{thm} 
  
\begin{proof}
We note by the same estimates that ${\bf K}$ is bounded from $L^p$ to $L^p$. It suffices to show that $u = {\bf K} f$ is the canonical solution. This is proved in Proposition 5.1 and Theorem 1.1 in \cite{DPZ}. For completeness, we sketch the proof here. First, let $\{\Omega^{(l)}\}$ be an increasing sequence of relatively compact subdomains of $\Omega$ with each $\Omega^{(l)}$ being the product of $C^2$ planar domains such that $\cup_l \Omega^{(l)} = \Omega$. Let $e^{(l)}, {\bf K}^{(l)}$ be defined accordingly on $\Omega^{(l)}$. Then it was showed in \cite{DPZ} (cf. equation (5.8)) using Kerzman's deep estimates of the Green function that
\begin{equation}\label{der}
  \frac{\partial^{s-1} \left(e^{(l)}\right)^{i_1, \cdots, i_s}_{i_j}(w, z)}{\partial \bar z_{i_1} \cdots \partial \bar z_{i_{j-1}} \partial \bar z_{i_{j+1}} \cdots \partial \bar z_{i_{s-1}}}  \rightarrow \frac{\partial^{s-1} e^{i_1, \cdots, i_s}_{i_j}(w, z)}{\partial \bar z_{i_1} \cdots \partial \bar z_{i_{j-1}} \partial \bar z_{i_{j+1}} \cdots \partial \bar z_{i_{s-1}}} 
  \end{equation}
   as $l \rightarrow \infty$.
Moreover, write $\Omega_\epsilon = \{z \in \Omega: \text{dist}(z, \partial \Omega) > \epsilon \}$. Since $f \in L^p_{(0, 1)}(\Omega)$ for any $p \geq 1$, by the standard mollification argument, there exists a sequence of $\bar\partial$-closed smooth $(0, 1)$-form  $f^\epsilon$ in $\Omega_\epsilon$, such that $f^\epsilon \rightarrow f$ in $L^p_{(0, 1)}(\Omega)$ for all $p >1$. Second, let $\chi $ be a $(n, n-1)$-form in $\Omega$ with compact support. Choosing $l$ sufficiently large and $\epsilon$ sufficiently small, such that  $supp(\chi) \subset \Omega^{(l)} \subset \overline{\Omega^{(l)}} \subset \Omega_\epsilon$, then $f^\epsilon \in C^\infty_{(0, 1)}(\overline{\Omega^{(l)}})$ and thus $u={\bf K}^{(l)} f^\epsilon$ is the canonical solution to $ \bar\partial u = f^\epsilon$ on $\Omega^{(l)}$ by Proposition \ref{smooth}. It follows that 
\begin{equation*}
\begin{split}
\int_{\Omega} \bar\partial {\bf K}f\wedge \chi &= - \int_{\Omega^{(l)}} {\bf K} f \wedge \bar\partial \chi = - \lim_{l \rightarrow \infty} \lim_{\epsilon \rightarrow 0}  \int_{\Omega^{(l)}} {\bf K}^{(l)} f^\epsilon \wedge \bar\partial \chi \\
&=  \lim_{l \rightarrow \infty} \lim_{\epsilon \rightarrow 0}  \int_{\Omega^{(l)}} f^\epsilon \wedge  \chi = \lim_{\epsilon \rightarrow 0}  \int_{\Omega} f^\epsilon \wedge  \chi = \int_\Omega f\wedge \chi,
\end{split}
\end{equation*}
where the second inequality follows from (\ref{der}).
This implies that $\bar\partial {\bf K}f = f$ weakly. Lastly, assume $g$ to be a $L^2$ holomorphic function in $\Omega $. Then similarly 
$$\int_\Omega {\bf K}f(z) \overline{g(z)} dV(z) =  \lim_{l \rightarrow \infty} \lim_{\epsilon \rightarrow 0}  \int_{\Omega^{(l)}} {\bf K}^{(l)} f^\epsilon(z) \overline{g(z)} dV(z) =  \lim_{l \rightarrow \infty} \lim_{\epsilon \rightarrow 0} 0 =0.$$
This means that $u = {\bf K} f$ is the canonical solution to $\bar\partial u = f$.
\end{proof}

\begin{rem}
When $\Omega$ is the product of star-shaped planar domains, for instance, $\Omega=\Delta^n$, there is no need to approximate $\Omega$ by $\Omega^{(l)}$ as above. 
By the standard mollification argument, any $L^p$ integrable $\bar\partial$-closed $(0, 1)$-form $f$ on $\Omega$ can be approximated in $L^p_{(0, 1)}(\Omega)$ by a sequence of $\bar\partial$-closed, smooth $(0, 1)$-forms $\{f^\epsilon\}$ on $\Omega$. Therefore that ${\bf K} f = \lim_{\epsilon \rightarrow0} {\bf K} f^\epsilon$ is the canonical solution is the direct consequence of the boundedness of ${\bf K}$  from $L^p$ to $L^p$. On the other hand, in the case of $\Omega=\Delta^n$, it can be verified directly without using Kerzman's estimates that 
$ \frac{\partial^{s-1} e^{i_1, \cdots, i_s}_{i_j}(w, z)}{\partial \bar z_{i_1} \cdots \partial \bar z_{i_{j-1}} \partial \bar z_{i_{j+1}} \cdots \partial \bar z_{i_{s-1}}} \in L^1(\Delta_{i_1} \times \cdots \times \Delta_{i_s})$ and the $L^1$ norm is also independent of $w$, and therefore ${\bf K}$ is bounded from $L^p$ to $L^p$ on $\Delta^n$.
\end{rem}

{\bf Acknowledgement:} The author would like to thank Zhenghui Huo for helpful discussions. The work was done when the author was visiting BICMR in Spring 2022. He thanks the center for providing him the wonderful research environment.

\end{document}